\documentclass[12pt,reqno]{amsart}

\usepackage{times}
\usepackage[T1]{fontenc}
\usepackage{mathrsfs}
\usepackage{latexsym}
\usepackage[dvips]{graphics}
\usepackage[dvips]{graphicx}
\usepackage{epsfig}
\usepackage{hyperref, amsmath,amsfonts,amsthm,amssymb,amscd, cleveref}
\usepackage{color}
\usepackage{tikz}
\usepackage{enumitem}

\newtheorem{thm}{Theorem}[section]

\newtheorem{cor}[thm]{Corollary}
\newtheorem{lem}[thm]{Lemma}
\newtheorem{prop}[thm]{Proposition}

\newtheorem{defn}[thm]{Definition}

\newtheorem{rmk}[thm]{Remark}

\numberwithin{equation}{section}

\newcommand{\comment}[1]{}
\newcommand{\bi}{\begin{itemize}}
\newcommand{\ei}{\end{itemize}}
\newcommand{\ben}{\begin{enumerate}}
\newcommand{\een}{\end{enumerate}}
\newcommand{\be}{\begin{equation}}
\newcommand{\ee}{\end{equation}}
\newcommand{\bea}{\begin{eqnarray}}
\newcommand{\eea}{\end{eqnarray}}
\newcommand{\bal}{\begin{align}}
\newcommand{\eal}{\end{align}}
\newcommand{\ba}{\begin{array}}
\newcommand{\ea}{\end{array}}

\DeclareFontFamily{U}{wncy}{}
\DeclareFontShape{U}{wncy}{m}{n}{<->wncyr10}{}
\DeclareSymbolFont{mcy}{U}{wncy}{m}{n}
\DeclareMathSymbol{\Sha}{\mathord}{mcy}{"58}

\begin{document}

\title{$2$-Superirreducibility of univariate polynomials over $\mathbb{Q}$ and $\mathbb{Z}$}
\author{Lara Du}



\date{\today}


\begin{abstract} 
This paper investigates whether or not polynomials that are irreducible over $\mathbb{Q}$ and $\mathbb{Z}$ can remain irreducible under substitution by all quadratic polynomials. It answers this question in the negative in the degree $2$ and $3$ cases and provides families of examples in both the affirmative and the negative categories in the degree $4$ case. Finally this paper explores what happens in higher degree cases, providing a family of examples in the negative category and offering a conjectured family for the positive category.
\end{abstract}

\maketitle


\section{Introduction} \label{intro}

\par The goal of this paper is to introduce and investigate a property of polynomials stricter than the familiar one of irreducibility. Consider a commutative integral domain $I$ and a polynomial $f\in I[x]$. When $k\in \mathbb{N}$, we say that $f$ is $k$\textit{-superirreducible} over $I$ when $f$ is irreducible and for all non-constant polynomials $g \in I[x]$ of degree at most $k$, the polynomial $f(g(x))$ remains irreducible. This definition requires irreducibility to hold simultaneously for all polynomials in a $(k+1)$-parameter family of polynomials with coefficients in $I$. In this paper, we will consider the cases $I=\mathbb{Z}$ and $I=\mathbb{Q}$.

\par The concept of superirreducibility over $\mathbb{Q}$  was studied by Schinzel \cite{schinzel} in 1967, though without giving it a name. Schinzel proved a criterion for non-existence which we consider below. Our definition extends this concept to the class of domains not having certain bad characteristics. We extend Schinzel's nonexistence criterion  as follows:

\begin{thm} \label{schinzelz}
For every $d\geq 3$ and domain $R$ of characteristic not dividing $d-1$, there are no $(d-1)$-superirreducibles over $R$ of degree $d$. 
If $R$ is a field of any characteristic, there are no $(d-1)$-superirreducibles over $R$ of degree $d$.
\end{thm}

\begin{cor} \label{deg3none}
There are no $d-1$-superirreducible polynomials of degree $d$ over $\mathbb{Q}$ or over $\mathbb{Z}$. 
\end{cor}

\par Corollary \ref{deg3none} tells us that in particular, there are no $2$-superirreducible polynomials of degree $3$ over $\mathbb{Q}$ or over $\mathbb{Z}$. Our focus in this paper will be primarily on the property of $2$-superirreducibility over $\mathbb{Q}$ and the corresponding slightly weaker property over $\mathbb{Z}$. In particular, we show that there are no $2$-superirreducible polynomials of degree smaller than $4$ over any domain $I$. As we demonstrate in Section \ref{sectiondegs1,2}, it is straightforward to show that there are no $2$-superirreducibles of degree $1$ or $2$ over $\mathbb{Z}$. For polynomials of degree $3$ or more, superirreducibility becomes a more complex question.

\par It is clear that if $f(x)$ is $2$-superirreducible over the fraction field $F$ of a domain $R$, then it must also be $2$-superirreducible over $R$. In general, going in the other direction, it is more difficult
to prove lack of 2-superirreducibility over $R$ than over $F$. For example, it was shown on page $13$ in \cite{BFMW} that 
$f(x)=x^4+x^2+2x+3$ is not $2$-superirreducible over the rationals since $g(x)=-\frac{1}{2}(x^2+x+3)$
causes $f(g(x))$ to factor over $\mathbb{Q}$. However, the authors confirmed that $f(g(x))$ is irreducible over the integers for any $g(x)=ax^2+bx+c \in \mathbb{Z}[x]$ with $|a|, |b|, |c|$ at most $1000$, providing some evidence in favour of $f(x)$ being $2$-superirreducible over $\mathbb{Z}$.

\par Next, in section \ref{sectionquartics} we exhibit families of $2$-superirreducible quartic polynomials over $\mathbb{Q}$.

\begin{thm} \label{x^4+a^4} \hfill
\begin{enumerate} [label=(\alph*)] \item  The polynomials $x^4+1$ and $x^4+2$ are $2$-superirreducible over the rationals. 
\item All the irreducible polynomials in the following sets are $2$-superirreducible over the rationals:
\begin{enumerate} [label=(\roman*)]
\item $\{x^4+p\}$ for $p$ prime and $p\equiv 7, 11 \pmod{16}$
\item $\{x^4+2p\}$ for $p$ prime and $p\equiv \pm 3 \pmod{8}$
\item $\{x^4+4p\}$ and $\{x^4-p\}$ for $p$ prime and $p\equiv \pm 3, -5 \pmod{16}$.
\end{enumerate}
\end{enumerate}
\end{thm}

\par A special case of Theorem \ref{x^4+a^4} part $(a)$ appeared as \cite[Question 11283]{abbot}. 

\begin{rmk}
For any $a\in \mathbb{Q}^{\times}$ we have that a polynomial $f$ is $2$-superirreducible if and only if $a^{\text{deg}f}f(x/a)$ is. Therefore part $(a)$ of Theorem \ref{x^4+a^4} implies $x^4+a^4$ and $x^4+2a^4$ are $2$-superirreducible for any $a\in \mathbb{Q}^{\times}$. Similar statements are true for part $(b)$.
\end{rmk}
\noindent In addition, we are able to prove later in the section that a larger family of quartics of the shape $ax^4+b$ is $2$-superirreducible (see Theorem \ref{genquartics}). We also prove that there are irreducible families of quartic polynomials which are not $2$-superirreducible.

\par Some examples of $2$-superirreducible polynomials of degree $6$ over $\mathbb{Q}$ may be found in Section $6$ of \cite{BFMW}, although they were not labelled as such. It seems quite likely that similar ideas would exhibit $2$-superirreducible polynomials of even degree $d\geq 6$ over $\mathbb{Q}$. However, new ideas are required to handle polynomials of odd degree. There are no known examples of $2$-superirreducible polynomials of odd degree. Therefore we devote Section \ref{w2superirred} to a discussion of a family of odd degree polynomials satisfying a slightly weaker notion of $2$-superirreducibility over $\mathbb{Z}$. 

\begin{thm} \label{deg5v1}
Let $k\geq 2$ and $f_k(x)=x^{2k+1}+2x+1$. Then whenever $a, b \in \mathbb{Z}$ with $a\neq 0$, the polynomial $f_k(ax^2+b)$ is irreducible over $\mathbb{Z}$.
\end{thm}

\par Superirreducibility over the rationals has been studied in the past, although not by name. For example in his paper \cite{schinzel}, Schinzel showed that a degree $d$ polynomial over $\mathbb{Q}$ cannot be $(d-1)$- superirreducible. More recently, Bober, Fretwell, Martin, and Wooley \cite{BFMW} considered superirreducibility as a potential limitation on (poly)smoothness of polynomial compositions. In addition, the author along with Bober, Fretwell, Kopp and Wooley counted the number of 2-superirreducible polynomials over finite fields and came up with asymptotic formulas in \cite{BDFKW}. One asymptotic formula was obtained by letting the degrees of the polynomials considered tend to infinity. Another was obtained while considering a sequence of finite fields having orders tending to infinity.
\\

\noindent \textbf{Acknowledgements:} This work has appeared as part of my PhD. thesis \cite{thesis} from the University of Michigan mathematics department, where I was partly supported by NSF grant DMS-1701577. I thank Trevor Wooley and Jeffrey Lagarias for helpful comments. I also thank the anonymous referee from a previous version of this paper for their useful and detailed comments.

\section{Lack of $2$-superirreducibility in degrees $1$ and $2$}\label{sectiondegs1,2}

\par It is trivial to see that a linear polynomial cannot be $2$-superirreducible over the rationals, nor over the integers. Concretely if we consider the problem over $\mathbb{Q}$ with $f(x)=ax+b$ and we then take $g(x)=x^2-\frac{b}{a}$ then $f(g(x))=ax^2$ is clearly reducible. Over $\mathbb{Z}$, showing that every degree $1$ polynomial has a quadratic substitution which makes it reducible is only a little more challenging. We find that if $f(x)=ax+b$ with $a,b$ integers and we take $g(x)=ax^2+(b+1)x$ then $f(g(x))=a^2x^2+a(b+1)x+b$ factors over $\mathbb{Z}$ as $(ax+b)(ax+1)$.

\par Increasing the degree by $1$, we next note that any quadratic polynomial $f(x)$ has an obstruction to being $2$-superirreducible. To be specific, we can extrapolate from the case when $f(x)$ was linear: if we take $g(x)=f(x)+x$ then we find that $f(g(x))$ is divisible by $f(x)$. We can generalise this observation in the following proposition:

\begin{prop} \label{degs1,2}
Let $R$ be a domain. If $f(x)$ is a degree $n$ polynomial with coefficients in $R$, then for any integer $k\geq 0$ we have that $f$ is not $(n+k)$-superirreducible over $R$.
\end{prop}

\begin{proof}
We can let $g(x)=x^kf(x)+x$. Then we have that $f(g(x))\equiv 0 \pmod {f(x))} $ so that $f(g(x))$ is divisible by $f(x)$ and so reducible.
\end{proof}

\par Explained from another angle, this proposition tells us that for any polynomial $f$ and any $n\geq \deg(f)$, we can find at least one polynomial $g(x)$ of degree $n$ such that the composition $f(g(x))$ is reducible. We note that for any given polynomial $f(x)$ there are many possible choices for $g(x)$ which we can use to prove Proposition \ref{degs1,2}. For example, instead of letting $g(x)=x^kf(x)+x$, we can replace $x^k$ by any polynomial of degree $k$.

\section{Lack of $2$-superirreducibles in degree $3$} \label{sectionschinzel}

\par In this section, we prove Theorem \ref{schinzelz}. This theorem is a slightly specialised version of a result of Schinzel's, modified from Schinzel's original argument. While Schinzel gives his result over $\mathbb{Q}$ and $\mathbb{Z}$, we will prove it for all unique factorization domains and their corresponding fields of fractions. It will then follow directly from our theorem that there can be no $2$-superirreducibles of degree $3$ over $\mathbb{Q}$ or over $\mathbb{Z}$. 

\begin{thm} \label{schinzelq}
For every $d\geq 3$ and any field $F$, there are no $(d-1)$-superirreducibles over $F$ of degree $d$.
\end{thm}

\begin{proof} 
Let $f(x)$ be an irreducible polynomial of degree $d$ with coefficients in $F$. Since we are working over a field, we can assume that $f$ is monic. We let $h(x)=x^df\left(\frac{1}{x}\right)$ and $g(x)=-\frac{h(x)}{x}+\frac{1}{x}$. We note that by adding the $\frac{1}{x}$ term, $g(x)$ is a polynomial. We also notice
\begin{align*}
f(g(x))&=f\left(-\frac{h(x)}{x}+\frac{1}{x}\right) \equiv 0 \pmod{h(x)}.
\end{align*}
This gives a factorisation of $f(g(x))$ over $F$ that has a degree $d-1$ factor.

\end{proof}

\par The proof that there are no $(d-1)$-superirreducibles of degree $d$ over a UFD $R$ of characteristic $0$ is much more involved than what we did to prove Theorem \ref{schinzelq}. Since we can no longer assume that $f$ is monic, if we were to replicate the argument in the proof of Theorem \ref{schinzelq} to construct $g$, we would have to divide by the leading coefficient $a_d$ of $f$. However this causes problems since then we are left with a polynomial that has coefficients in $F$ and not $R$. Therefore we must figure out a way to rescale and shift things to ensure that $g(x)$ is in $R[x]$.

\begin{proof} [Proof of Theorem \ref{schinzelz}]

\par Suppose that $f$ is irreducible over $R$ but not necessarily monic. Let $f(x)=a_dx^d+a_{d-1}x^{d-1}+\dots +a_1x+a_0$. Drawing inspiration from the monic case in the proof of Theorem \ref{schinzelq}, we let $h(x)=x^d f\left(\frac{1}{x}\right)$. Next we define $g(x)$ (our polynomial of degree $d-1$ to be substituted into $f$) as $g(x)=-\frac{h(x)}{a_d x}+\frac{1}{x}$. We note that by adding the $\frac{1}{x}$ term, $g(x)$ is a polynomial. We also notice
\begin{align*}
f(g(x))&=f\left(-\frac{h(x)}{a_d x}+\frac{1}{x}\right) \equiv 0 \pmod{h(x)}.
\end{align*}
This is a factorisation of $f(g(x))$ over $F$ where one of the factors has degree $d-1$.

\par To get a factorisation over $R$, we need to rescale and shift things. We see that the $g(x) \in F[x]$ that we previously chose is not a polynomial in $R[x]$ because of the division by $a_d$. In order to account for this problem, we would like to introduce a scaling constant to help us clear denominators:
$$A:=(d-1)^2a_d^2.$$
We note that $A\neq 0$ because the characteristic of $R$ does not divide $d-1$. If we simply rescale by $A$ and let $\tilde{h}(x)=h(Ax)$ and $\tilde{g}(x)=g(Ax)$, then this clears most denominators except for one: the constant term of $\tilde{g}(x)$ is $-\frac{a_{d-1}}{a_d}$, which is in $F$ and not $R$. To account for this term, we need to first introduce a shift and then rescale by $A$. Our shift necessitates the presence of the $(d-1)^2$ term in A because of a $\frac{d}{d-1}-1$ term that will arise as a result of the binomial theorem. Explicitly our shift is:
$$H(x)=x^df\left(\frac{1}{x}-\frac{a_{d-1}}{(d-1)a_d}\right), \qquad G(x)=-\frac{H(x)}{a_dx}+\frac{1}{x}-\frac{a_{d-1}}{(d-1)a_d}$$
and we take the rescaled versions of these functions to be
$$\tilde{H}(x)=H(Ax), \qquad \tilde{G}(x)=G(Ax).$$
Calculating explicitly, we find that
\begin{equation*}
\tilde{H}(x)=A^dx^d \sum_{i=0}^{d} a_i\left(\frac{1}{Ax}-\frac{a_{d-1}}{(d-1)a_d}\right)^i
\end{equation*}
and 
\begin{align*}
\tilde{G}(x)&=-\frac{\tilde{H}(x)}{a_dAx }+\frac{1}{Ax}-\frac{a_{d-1}}{(d-1)a_d}\\
&=-A^{d-2}(d-1)^2a_dx^{d-1} \sum_{i=0}^{d} a_i\left(\frac{1}{Ax}-\frac{a_{d-1}}{(d-1)a_d}\right)^i+\frac{1}{Ax}-\frac{a_{d-1}}{(d-1)a_d}. 
\end{align*}
We note that the factor $-A^{d-2}(d-1)^2a_dx^{d-1}$ on the outside of the summation clears almost all denominators in the summation. Specifically the only terms whose denominators are not cleared are
$$\frac{a_d}{A^dx^d}-\frac{da_{d-1}}{(d-1)A^{d-1}x^{d-1}}+\frac{a_{d-1}}{A^{d-1}x^{d-1}}.$$
Multiplying by the  $-A^{d-2}(d-1)^2a_dx^{d-1}$ on the outside of the summation, we receive 
$$-\frac{1}{Ax}+\frac{da_{d-1}}{a_d(d-1)}-\frac{a_{d-1}}{a_d}$$
which is exactly cancelled by the terms added on at the end, after the summation. Therefore $\tilde{G}(x)$ is a polynomial in $R[x]$, as desired. Finally we observe that $f(\tilde{G}(x))$ is reducible since
\begin{align*}
f(\tilde{G}(x))&=f\left(-\frac{\tilde{H}(x)}{a_dAx }+\frac{1}{Ax}-\frac{a_{d-1}}{(d-1)a_d}\right)\\
&\equiv f\left(\frac{1}{Ax}-\frac{a_{d-1}}{(d-1)a_d}\right) \pmod{\tilde{H}(x)}\\
& \equiv 0 \pmod{\tilde{H}(x)}
\end{align*}
so that $f(\tilde{G}(x))$ is divisible by $\tilde{H}(x)$ over the field of fractions of $R$. Since $R$ is a UFD, we can use Gauss' Lemma to conclude that $f(\tilde{G}(x))=\tilde{H}(x)k(x)$ for some polynomial $k(x)$ in $F[x]$. Scaling $\tilde{H}(x)$ and $k(x)$ appropriately by some constants $c_1$ and $c_2$, we see that $f(\tilde{G}(x))=(c_1\tilde{H}(x))(c_2k(x))$ for polynomials $c_1\tilde{H}(x)$ and $c_2k(x)$ in $R[x]$. Therefore $f(\tilde{G}(x))$ is reducible over $R$.
\end{proof}

\par We then receive Corollary \ref{deg3none} as a direct consequence of Theorem \ref{schinzelz}. We note that Theorem \ref{schinzelz} leaves open the possibility that over a domain of characteristic dividing $d-1$ there might exist $(d-1)$-superirreducibles of degree $d$. It would be interesting to determine whether or not $2$-superirreducible cubic polynomials exist over a characteristic $2$ domain such as $\mathbb{F}_2[t]$.

\section{An equivalent definition of $2$-superirreducibility}

\par Theorem \ref{degs1,2} cannot be applied to prove $2$-superirreducibility of polynomials of degree greater than $2$, so we must think about $2$-superirreducibility in other ways. In this section, we state a lemma equating $2$-superirreducibility with the non-existence of solutions to a certain equation. This lemma will be used to prove Theorem \ref{eventerms}.

\begin{lem} \label{fact}
Let $f(x) \in \mathbb{Q}[x]$ be a monic irreducible polynomial of degree $d$, and let $\alpha$ be a root of $f(x)$ in the splitting field of $f(x)$ over $\mathbb{Q}$. If $g(t)$ is a quadratic polynomial, the composition $f(g(t))$ factors over $\mathbb{Q}$ if and only if $g(t) = \alpha$ has a solution in $\mathbb{Q}(\alpha)$.
\end{lem}

This Lemma is a consequence of Theorem $20$ in \cite{schinzel3}. For completeness, we state the theorem here in the language of this paper. 

\begin{thm}\label{twenty}
Let $f(x)\in \mathbb{Q}[x]$ be a monic irreducible polynomial and suppose that $\alpha$ is a root of $f(x)$ in the splitting field of $f(x)$ over $\mathbb{Q}$. Whenever $g(t)-\alpha$ has the factorisation $\prod_{p=1}^{r} h_p(t)^{e_p}$, as a product of distinct monic irreducible polynomials $h_p\in \mathbb{Q}(\alpha)$, with each $e_p\geq 1$, then $f(g(t))$ has the factorisation $\prod_{p=1}^{r} N_{\mathbb{Q}(\alpha)/\mathbb{Q}}\left(h_p(t)\right)^{e_p}$ where each $N_{\mathbb{Q}(\alpha)/\mathbb{Q}}\left(h_p(t)\right)\in \mathbb{Q}[t]$ is monic and irreducible.
\end{thm}

We conclude the following corollary as an immediate consequence of this theorem.

\begin{cor}
Let $f(x) \in \mathbb{Q}[x]$ be a monic irreducible polynomial of degree $d$, and let $\alpha$ be a root of $f(x)$ in the splitting field of $f(x)$ over $\mathbb{Q}$.  Let $g(t)\in \mathbb{Q}[t]$ be quadratic. Then the following hold.
\begin{enumerate}
\item If $g(t)-\alpha$ is irreducible over $\mathbb{Q}(\alpha)$, then $f(g(t))$ is irreducible over $\mathbb{Q}$.
\item If $g(t)-\alpha$ splits over $\mathbb{Q}(\alpha)$, then $f(g(t))$ splits over $\mathbb{Q}$.
\end{enumerate}
\end{cor}

\section{$2$-superirreducible families in degree $4$}\label{sectionquartics}

\par We begin by getting an idea of what a $2$-superirreducible degree $4$ polynomial must look like. This is a fairly straightforward problem over the rationals, but is more challenging over the integers.

\begin{thm} \label{eventerms}
If $f(x)$ is a polynomial over $\mathbb{Q}$ of degree $2N$, its linear term has nonzero coefficient and all of its other odd degree terms have zero coefficients, then $f(x)$ is not $N$-superirreducible.
\end{thm}

\begin{proof}
The hypotheses allow us to assume that $f(x)$ can be written in the shape $f(x)=ax+F(x^2)$, where $a\neq 0$ and $F(x)\in \mathbb{Q}[x]$ has degree $N$. Let $\alpha$ be a root of $f$ in its splitting field and set $\beta=\alpha^2$. Then $\alpha=-a^{-1}F(\beta)$. If we now put $g(x)=-a^{-1}F(x)$, then we see that $\beta$ is a root of $g(x)-\alpha$ in $\mathbb{Q}(\alpha)$ and so by Lemma \ref{fact} we find that $f(g(x))$ splits over $\mathbb{Q}$.
\end{proof}

\begin{cor} \label{quarticsimple}
Let $f(x)=x^4+ax^3+bx^2+cx+d$ be a monic quartic polynomial with rational coefficients. If $8c-4ab+a^3\neq 0$, then $f(x)$ cannot be $2$-superirreducible over the rationals.
\end{cor}

\begin{proof}
By completing the fourth power, we can assume $f$ has a zero coefficient for its third degree term. The condition $8c-4ab+a^3\neq 0$ ensures that the linear term has nonzero coefficient. We conclude by applying Theorem \ref{eventerms}.
\end{proof}

\par We note that in the last corollary, since we were considering superirreducibility over $\mathbb{Q}$, if $f$ was not monic, we could still divide out by the leading coefficient and apply the corollary. It seems challenging to generalise this to a statement over $\mathbb{Z}$. 

\par The simplest family of quartics that are not of the form given in Corollary \ref{quarticsimple} are those of the shape $ax^4+b$. Indeed we can prove using the following two lemmas that certain quartics of this shape are $2$-superirreducible.

\begin{lem} The equation $X^4+Y^4=Z^2$ has no non-trivial rational solutions. \label{x^4+y^4=z^2}
\end{lem}

\begin{proof}
It is well-known that the equation $X^4+Y^4=Z^2$ has no integer solutions unless $XYZ=0$ (see Theorem $226$ in \cite{hardyandwright}). The argument for the integer case extends simply to the rational case (stated as Theorem $1$ on page $16$ of \cite{mordell}).
\end{proof}

\par We can use a similar infinite descent argument as in the proof of Lemma \ref{x^4+y^4=z^2} to show that \\$X^4+2Y^4=Z^2$ has no non-trivial rational solutions and then conclude using Lemma \ref{x^4+y^4=z^2} that all polynomials of the form $x^4+2a^4$, with $a\in \mathbb{Q}$ are $2$-superirreducible over $\mathbb{Q}$. The proof that $X^4+2Y^4=Z^2$ has no non-trivial rational solutions is summarised in \cite{legendre} on page $405$ and as an exercise on page $17$ in \cite{carmichael}. The proof is a standard application of the use of descent in analysing Diophantine equations. The sources listed are by now somewhat obscure and difficult to find. Since the proof is not too long and it would be tedious for the reader to reconstruct all the details, we present it briefly here for the sake of completeness. 

\begin{lem} The equation $X^4+2Y^4=Z^2$ has no non-trivial rational solutions. \label{x^4+2y^4=z^2}
\end{lem}

\begin{proof} To begin with, we suppose there exists a solution $(x,y,z)$ to the equation $X^4+2Y^4=Z^2$, where $x,y,z$ are the smallest positive integers which satisfy the equation. Therefore $x,y,$ and $z$ must be pairwise coprime and in addition, $x$ and $z$ must be odd and $y$ must be even. We can write $z=x^2+\frac{2p}{q}y^2$ for some coprime integers $p,q$. Since $z^2=x^4+2y^4$ we see that $z>x^2$ and so we can assume that both $p$ and $q$ are positive.

\par Substituting $z=x^2+\frac{2p}{q}y^2$ into $x^4+2y^4=z^2$, we get that $\frac{x^2}{y^2}=\frac{q^2-2p^2}{2pq}$. Since $(x,y)=(p,q)=1$ and $p$ and $q$ are positive, we see that $y^2=2pq$ and $x^2=q^2-2p^2$, so that $z^2=q^2+2p^2$. We note that since $(p,q)=1$ and $x^2=q^2-2p^2$, we have that $(q,x)=1$.  Also since $x$ is odd, we must have that $q$ is odd and $p$ is even. We can rewrite $p$ as $2r$ and rearrange $x^2=q^2-2p^2$ to get $2r^2=\left(\frac{q+x}{2}\right)\left(\frac{q-x}{2}\right)$, where $\left(\frac{q+x}{2},\frac{q-x}{2}\right)$ is either $(2m^2, n^2)$ or $(m^2, 2n^2)$. Since $2p^2=q^2-x^2$ and $p,m,$ and $n$ are all positive, we see that in the respective cases we have $(q,x,p)$ is either $(2m^2+n^2, 2m^2-n^2, 2mn)$ or $(m^2+2n^2, m^2-2n^2, 2mn)$.

\par In either case, we can substitute $p$ and $q$ into $y^2=2pq$ and receive either $y^2=4mn(2m^2+n^2)$ or $y^2=4mn(m^2+2n^2)$. We work with $y^2=4mn(m^2+2n^2)$. The argument with $y^2=4mn(2m^2+n^2)$ is very similar. We observe that since $p$ and $q$ are coprime, $m$ and $n$ must also be coprime. This tells us that there are integers $f,g,h$ with $m=f^2, n=g^2$ and $m^2+2n^2=h^2$. We then have $f^4+2g^4=h^2$, yielding the solution $(f,g,h)$ to $X^4+2Y^4=Z^2$ where $h<z$. 

\par Since $X^4+2Y^4=Z^2$ has no non-trivial integer solutions, by the ideas underlying the proof of Lemma \ref{x^4+y^4=z^2}, it also has no non-trivial rational solutions.
\end{proof}

\begin{rmk} \label{rmkkarpilovsky} A criterion for $x^n-c$ to be reducible over a field $F$ is given as Theorem $1.6$ in \cite{karpilovsky}. Translating that to our specific case, we can say that $x^{4k}-c$ is reducible over $\mathbb{Q}$ if and only if either $-4c$ is a rational fourth power, or $c$ is a $p$th power for some prime $p$ dividing $4k$.

\end{rmk}

\par We can use the above two lemmas, as well as Remark \ref{rmkkarpilovsky} to prove part $(a)$ of Theorem \ref{x^4+a^4}.

\begin{proof} [Proof of Theorem \ref{x^4+a^4} part $(a)$] 
Let $f(x)=x^4+1$. We'll show that $f(x)$ is $2$-superirreducible over $\mathbb{Q}$. Let $\alpha$ be a root of $f$ in its splitting field. Let $g(x)=bx^2+cx+d$ with $b,c,d \in \mathbb{Q}$. Suppose that $f(g(x))$ is not irreducible and hence factors over the rationals as the product of two non-constant polynomials, say $h_1(x)$ and $h_2(x)$.

\par Let $\beta$ be a root of $g(x)-\alpha$ in $\mathbb{Q}(\alpha)$. We observe that $\mathbb{Q}(\alpha)=\mathbb{Q}(g(\beta))$ so that $\mathbb{Q}(\alpha)$ is a subfield of $\mathbb{Q}(\beta)$. However $g(\beta)=\alpha$ implies that $\beta$ must be a root of either $h_1$ or $h_2$, so that $[\mathbb{Q}(\beta): \mathbb{Q}]=4$. This tells us that $\mathbb{Q}(\alpha)=\mathbb{Q}(\beta)$, so we conclude from Theorem \ref{twenty} that $h_1(x)$ and $h_2(x)$ are both degree $4$ polynomials.

\par Using the fact that $g(\beta)=\alpha$, we have that $(2b\beta+c)^2=4b\alpha+(c^2-4bd)$, which we can rewrite as $m\alpha+n$ with $m,n$ rationals. Taking norms from $\mathbb{Q}(\alpha)$ down to $\mathbb{Q}$ we find:
\begin{align*}
(N(2b\beta+c))^2&=N(m\alpha+n)=m^4+n^4.
\end{align*}
Since $2b\beta+c$ is a nonzero algebraic number, its norm must be a nonzero rational number. This gives us the equation $m^4+n^4=z^2$ where $m,n$ and $z$ are all rationals with $m$ and $z$ nonzero. 

\par From Lemma \ref{x^4+y^4=z^2}, we find that we must have $n=0$. However this would mean that $g$ has a repeated root and so can be written as $g(x)=r(sx-t)^2$ for some $s, t \in \mathbb{Z}$ and $r\in \mathbb{Q}$. This would then mean $f(g(x))=r^4(sx-t)^8+1$ and so we can write $\frac{f(g(x))}{r^4}$ as $X^8+R^4$, where $R=\frac{1}{r}$. We can now apply Remark \ref{rmkkarpilovsky} with $c=-R^4$ and $k=2$. Remark \ref{rmkkarpilovsky} now tells us that $X^8+R^4$ is reducible over $\mathbb{Q}$ if and only if either $4R^4$ is a rational $4$th power, or $-R^4$ is a rational power of $2$. We can see that neither of these can be true, so $X^8+R^4$ must be irreducible over $\mathbb{Q}$.

\par Since $f(g(x))$ does not have a factorisation over $\mathbb{Q}$, we see that $f(x)$ is $2$-superirreducible over the rationals. We can now rerun the argument above and apply Lemma \ref{x^4+2y^4=z^2} to conclude that $x^4+2$ must also be $2$-superirreducible over $\mathbb{Q}$. 
\end{proof}

\par The referee from an earlier version of this paper has generously pointed out that when $C\in \mathbb{Q}^{\times}$ and the elliptic curve $y^2=x^3+Cx$ has only two rational points, then the polynomial $f(x)=x^4+C$ is $2$-superirreducible. We are able to identify such elliptic curves using the database on the lmfdb.org website.

\par Next we observe that the technique from the proof of Theorem \ref{x^4+a^4} does not work to show that $f(x)=x^4+3$ is 2-superirreducible because the equation $m^4+3n^4=z^2$ has nontrivial integer solutions (e.g. $m=1, n=1, z=2$). Similarly $f(x)=x^4+5$ gives rise to the equation $m^4+5n^4=z^2$ which has $m=1, n=2, z=9$ as a solution. 

\begin{proof} [Proof of Theorem \ref{x^4+a^4} part $(b)$]  
We can use Theorem \ref{rmkkarpilovsky} to check for irreducibility.  We can then run the argument from the proof of Theorem \ref{x^4+a^4} part $(a)$ and use the argument on page $23$ of \cite{mordell}, which shows lack of non-trivial integer solutions for  families of the form $x^4+dy^4=z^2$ for 
\begin{enumerate}
\item $d=p$ for $p\equiv 7, 11 \pmod{16}$
\item $d=2p$ for $p\equiv \pm 3 \pmod{8}$
\item $d=4p$ and $d=-p$ for $p\equiv \pm 3, -5 \pmod{16}$
\end{enumerate}
to conclude $2$-superireducibility.
\end{proof}

\begin{rmk} \label{others}
We can apply Theorem \ref{x^4+a^4} part $(b)$ to conclude that $x^4-3, x^4+6$ and $x^4+7$ are all examples of $2$-superirreducible polynomials. 
\end{rmk}

\begin{rmk}
Theorem \ref{rmkkarpilovsky} and Theorem \ref{x^4+a^4} part $(b)$ together establish a way to check for $2$-superirreducibility for selected cases of quartic polynomials. The families given in Theorem \ref{x^4+a^4} part $(b)$ are unlikely to form an exhaustive list.
\end{rmk}

\begin{rmk} \label{deg4best}
If an integer is a $4k$-th power, then it is automatically a $4$th power. Therefore the method of proof of Theorem \ref{x^4+a^4} may be adapted to show that all the polynomials in the families $\{x^{4k}+a^4\}$ and $\{x^{4k}+2a^4\}$ are $2$-superirreducible over the rationals. We can also construct other $2$-superirreducible families, such as the ones Remark \ref{others} gives rise to: $\{x^{4k}-3a^4\}$, $\{x^{4k}+6a^4\}$ and $\{x^{4k}+7a^4\}$.
\end{rmk}

\par We can also say something about $2$-superirreducibility of a large family of non-monic polynomials of degree $4k$ where $k\in \mathbb{N}$.

\begin{thm} \label{genquartics}
Let $a$ and $b$ be coprime integers. The polynomial $f(x)=ax^{4k}+b$ is $2$-superirreducible over $\mathbb{Z}$ if it is irreducible and $a$ is a non-square mod $b$, or $b$ is a non-square mod $a$.
\end{thm}

\begin{proof}
We first note that $ax^{4k}+b$ is almost always irreducible over the integers. The specific criterion for $ax^{4k}+b$ to be irreducible is given in Remark \ref{rmkkarpilovsky}. We note that the method of proof of Theorem \ref{x^4+a^4} part $(a)$ applies here. We are then done by Legendre's Theorem (presented in \cite{irelandrosen} as Proposition $17.3.1$), which says that the equation $ax^2+by^2-z^2=0$ has no nontrivial integer solutions unless $a$ is a square mod $b$ and $b$ is a square mod $a$. 
\end{proof}

\par We would like to extend the results in this section to polynomials of the form $x^n+a^n$, where $n>4$. Darmon and Merel in \cite{darmonmerel} and Poonen in \cite{poonen} proved that equations of the form $x^n+y^n=z^2$  have no non-trivial primitive integer solutions if $n\geq 3$. Therefore it may seem likely that the ideas in the proof of Theorem \ref{x^4+a^4} would extend to higher degrees. Unfortunately this cannot happen because $x^n+a^n$ fails to be irreducible when $n$ is not a power of $2$, so it does not make sense to talk about $2$-superirreducibility. If $n$ is divisible by an odd prime $p$, then $x^n+a^n$ is divisible by $x^{n/p}+a^{n/p}$. For example
$$x^6+a^6=(x^2+a^2)(x^4-a^2x^2+a^4).$$
Despite this, it is possible to find families of even degree $2$-superirreducibles that are not of the form given in Remark \ref{deg4best}:

\begin{rmk}
On lmfdb.org, the elliptic curve $y^2=x^3+6$ has Cremona label $15552bp1$. It has rank $0$ and trivial torsion and so has no rational points on it except the point at infinity. Therefore there are no rational solutions to the equation $y^2=x^6+6$. Using similar methods as in the proof of Theorem \ref{x^4+a^4} part $(a)$ we may show that the polynomials in the family $\{x^{6k}+6a^{6k}\}$ are $2$-superirreducible over the rationals. It is also possible to find other analogous families of $2$-superirreducibles over $\mathbb{Q}$ by considering elliptic curves with no rational points on them. This was done in the discussion surrounding equation $6.1$ in \cite{BFMW}
\end{rmk}

\section{A family of weakly $2$-superirreducible polynomials} \label{w2superirred}

\par A natural question to ask is whether there exists an odd degree $2$-superirreducible polynomial over $\mathbb{Q}$. We know by Corollary \ref{deg3none} that if such an odd degree polynomial exists, it must have degree at least $5$. In this paper, we are unable to solve this problem completely. However we are able to find a family of irreducible odd degree polynomials that remain irreducible over $\mathbb{Z}$ under any substitution of form $ax^2+b$.

\begin{defn}
Let $R$ be a commutative domain. The polynomial $f(x)\in R[x]$ is weakly $k$-superirreducible over $R$ if for all $g(t)\in R[t]$ of form $ax^j+b$ with $a,b$ in $R$ and $j\leq k$, the composition $f(g(t))$ is irreducible over $R$.
\end{defn}

\par We now state the main theorem of this section, which is analogous to Theorem \ref{deg5v1} and which we will prove at the end of the section.

\begin{thm} \label{mainweak}
For $k\geq 2$, all the polynomials $x^{2k+1}+2x+1$ are weakly $2$-superirreducible over the integers.
\end{thm}

\par Perron's criterion (page $291-292$ of \cite{perron}) establishes irreducibility for all the polynomials in the family (so they remain irreducible under linear substitutions). However Theorem \ref{mainweak} shows that they also remain irreducible under a substitution of the form $g(x)=ax^2+b$. To establish Theorem \ref{mainweak}, we aim to understand the ring of algebraic integers in the splitting field of $x^{2k+1}+2x+1$. Theorem $4$ in \cite{trinomialdiscriminant} gives a general formula for the discriminant of the trinomial $f(x)=x^n+ax^k+b$ with $0<k<n$. Specialising this to our case, we have that the discriminant of the polynomial $f(x)=x^{2k+1}+2x+1$ is
$$(-1)^{k}\left((2k+1)^{2k+1}+2^{2k+1}(2k)^{2k}\right)$$ 
which is an odd integer.
We receive the following Corollary by Proposition $9.1$ in \cite{janusz}.

\begin{cor} \label{algint}
Let $\theta$ be a root of $f(x)=x^{2k+1}+2x+1$. Since the discriminant of $x^{2k+1}+2x+1$ is odd, all elements of the ring of integers of the splitting field for $f$ take the form
$$\frac{a_0+a_1\theta+\dots+a_{2k}\theta^{2k}}{m}$$
with $a_i \in \mathbb{Z}$, for some fixed odd $m$.
\end{cor}

We are now ready to prove Theorem \ref{mainweak}. 

\begin{proof}
Let $f(x)=x^{2k+1}+2x+1$ and $g(x)=ax^2+c$ for fixed $a, c \in \mathbb{Z}$. Then
$$f(g(x))=(ax^2+c)^{2k+1}+2(ax^2+c)+1.$$
For this composition to be reducible, we need $ax^2=\theta-c$ to be solvable in the splitting field for $f$. Let $\beta$ be a solution. We note
on writing $a\beta^2=\theta-c$, that $a^2\beta^2=a\theta-ac.$
By replacing $a\beta$ with $\beta_0$, we see that $\beta_0^2=M\theta+N$ is an algebraic integer with $M$ and $N$ integers satisfying $M|N$.

\par By Corollary \ref{algint}, we can write $\beta_0=\tilde{A}_{2k}\theta^{2k}+\tilde{A}_{2k-1}\theta^{2k-1}+\dots+\tilde{A}_1\theta+\tilde{A}_0$, with each $\tilde{A}_i=\frac{A_i}{m}$, where $A_i$ is an integer and $m$ is the fixed odd number from Corollary \ref{algint}. Then
$$\left(\tilde{A}_{2k}\theta^{2k}+\tilde{A}_{2k-1}\theta^{2k-1}+\dots+\tilde{A}_1\theta+\tilde{A}_0\right)^2=M\theta+N.$$
If we multiply both sides of this equation by $m^2$, we receive
\begin{align}
\left(A_{2k}\theta^{2k}+A_{2k-1}\theta^{2k-1}+\dots+A_1\theta+A_0\right)^2=M_0\theta+N_0 \label{pikachu}
\end{align}
where the $A_i$'s are all integers and $M_0$ and $N_0$ are integers with $M_0$ dividing $N_0$.

\par We will show that the only integral solution to equation \ref{pikachu} is $A_i=0$ for all $i$ and $M_0=N_0=0$. We can begin by assuming that $(A_0, A_1, \dots A_{2k})=1$. This is because if $d|A_i$ for all $i$ and $d>1$, then necessarily $d^2|M_0$ and $d^2|N_0$ and we could divide both sides of equation \ref{pikachu} by $d^2$ and obtain a smaller solution. We make use of the relations
\begin{align}
\theta^{2k+j}=-2\theta^j-\theta^{j-1} \qquad (\text{for}\ 1\leq j \leq 2k) \label{charizard}
\end{align}
which are valid because $\theta$ is a root of $f$. Expanding the left hand side of equation \ref{pikachu} we see
\begin{align*}
\left(A_{2k}\theta^{2k}+A_{2k-1}\theta^{2k-1}+\dots+A_1\theta+A_0\right)^2&\equiv A_{2k}^2\theta^{4k}+\dots+A_1^2\theta^2+A_0^2 \pmod 2\\
&\equiv -(A_{2k}^2\theta^{2k-1}+A_{2k-1}^2\theta^{2k-3}+\dots+A_{k+1}^2\theta)\\
&\qquad \qquad +(A_{k}^2\theta^{2k}+\dots+A_1^2\theta^2+A_0^2) \pmod 2
\end{align*}
where we received the second congruence by using relation \ref{charizard}. If we now equate all this to the right hand side of equation \ref{pikachu}, we see that 
\begin{align}
A_{k+1}^2\theta+A_0^2 \equiv M_0\theta+N_0 \pmod 2 \label{jigglypuff}
\end{align}
$$A_1^2\equiv \dots \equiv A_k^2 \equiv A_{k+2}^2\equiv \dots \equiv A_{2k}^2\equiv 0 \pmod 2.$$

\par We now know that $A_i\equiv 0 \pmod 2$ for all $i$, except possibly $i\in \{0, k+1\}$. To show that $A_0\equiv A_{k+1}\equiv 0 \pmod 2$, we examine the coefficient of $\theta^2$ in equation \ref{pikachu}, this time working modulo $4$. We find that
\begin{align*}
0&= 2A_0A_2+A_1^2-2A_{k+1}^2-4 \sum_{\substack{i+j=2k+2 \\ i<j}} A_iA_j-2 \sum_{\substack{i+j=2k+3 \\ i<j}} A_iA_j\\
&\equiv 2A_0A_2-2A_{k+1}^2-2A_{k+1}A_{k+2} \pmod 4\\
&\equiv -2A_{k+1}^2 \pmod 4.
\end{align*}
This tells us that $A_{k+1}\equiv 0 \pmod 2$, so congruence \ref{jigglypuff} now becomes $N_0\equiv A_0^2 \pmod 2$. But we also have $M_0|N_0$ and $M_0\equiv A_{k+1}^2 \equiv 0 \pmod 2$, so we conclude $A_0\equiv 0 \pmod 2$. This tells us that all the $A_i$ are even and contradicts the fact that $(A_0, A_1, \dots A_{2k})=1$.

\end{proof}

\par It may be tempting to suppose that large Galois groups should correlate with $2$-superirreducibility and small Galois groups with the absence of $2$-superirreducibility. Indeed a computation in Magma shows for $5\leq 2n+1\leq 99$ that $x^{2n+1}+2x+1$ has Galois group $S_{2n+1}$. However the $2$-superirreducible polynomial $x^4+16$ has the decidedly small Galois group $C_2\times C_2$ and the cubic polynomial $x^3+2x+1$, which is not $2$-superirreducible has the largest possible Galois group $S_3$.



%


\end{document}